\documentclass{article}
\usepackage{amsmath, amsthm, amscd, amsfonts, amssymb, graphicx, color}
\usepackage[bookmarksnumbered, colorlinks, plainpages]{hyperref}

\newtheorem{theorem}{Theorem}[section]
\newtheorem{lemma}[theorem]{Lemma}

\newtheorem{corollary}[theorem]{Corollary}
\theoremstyle{definition}
\newtheorem{definition}[theorem]{Definition}

\theoremstyle{remark}
\newtheorem{remark}[theorem]{Remark}
\numberwithin{equation}{section}

\newcommand{\ket}[1]{ | #1  \rangle}
\newcommand{\bra}[1]{ \langle #1  |}
\newcommand{\braket}[2]{ \left\langle #1  | #2  \right\rangle}
\newcommand{\ii}{\mathbf{i_1}}
\newcommand{\iii}{\mathbf{i_2}}
\newcommand{\jj}{\mathbf{j}}
\newcommand{\ee}{\mathbf{e_1}}
\newcommand{\eee}{\mathbf{e_2}}
\newcommand{\e}[1]{\mathbf{e_{#1}}}
\newcommand{\ik}[1]{\mathbf{i_{#1}}}

\newcommand{\M}{\mathbb{M}}
\newcommand{\C}{\mathbb{C}}
\newcommand{\D}{\mathbb{D}}
\newcommand{\R}{\mathbb{R}}
\newcommand{\nc}{\mathcal{NC}}
\renewcommand{\(}{\left(}
\renewcommand{\)}{\right)}
\newcommand{\oa}{\left\{}
\newcommand{\fa}{\right\}}
\renewcommand{\[}{\left[}
\renewcommand{\]}{\right]}
\renewcommand{\P}[2]{P_{#1}( #2 )}

\newcommand{\bo} {\ensuremath{{\bf i_1}}}

\newcommand{\bt} {\ensuremath{{\bf i_2}}}

\newcommand{\eo} {\ensuremath{{\bf e_1}}}
\newcommand{\et} {\ensuremath{{\bf e_2}}}

\newcommand{\h}[1]{{\widehat{#1}}}
\newcommand{\hh}{{\widehat{1}}}
\newcommand{\hhh}{{\widehat{2}}}
\newcommand{\mC}{\ensuremath{\mathbb{C}}}

\newcommand{\mo}{\mathbf{1}}
\newcommand{\mt}{\mathbf{2}}
\newcommand{\mk}{\mathbf{k}}
\newcommand{\scalarmath}[2]{\( #1, #2 \)}
\begin{document}
\vspace{4cm}
\begin{center} \LARGE{\textbf{Infinite Dimensional Bicomplex Spectral Decomposition Theorem}}
\end{center}
\vspace{1cm}

\begin{center} \bf{K.S. Charak$^{1}$, R. Kumar$^{2}$ and D. Rochon$^{3}$}
\end{center}

\smallskip

\begin{center}{$^{1}$ Department of Mathematics, University of Jammu,\\
Jammu-180 006, INDIA.\\
E-mail: kscharak7@rediffmail.com }
\end{center}

\begin{center}{$^{2}$ Department of Mathematics, University of Jammu,\\
Jammu-180 006, INDIA.\\
E-mail: ravinder.kumarji@gmail.com }
\end{center}

\begin{center} {$^{3}$ D\'epartement de math\'ematiques et d'informatique,\\
Universit\'e du Qu\'ebec \`a Trois-Rivi\`eres, C.P. 500 Trois-Rivi\`eres, Qu\'ebec, Canada G9A 5H7.  \\
E-mail: Dominic.Rochon@UQTR.CA,\\
Web: www.3dfractals.com}
\end{center}

\medskip
\begin{abstract}
 This paper presents a bicomplex version of the Spectral Decomposition Theorem on infinite dimensional bicomplex Hilbert spaces. In the process, the ideas of bounded linear operators, orthogonal complements and compact operators on bicomplex Hilbert spaces are introduced and treated in relation with the classical Hilbert space $M'$ imbedded in any bicomplex Hilbert space $M$.
\end{abstract}

\noindent \textbf{AMS [2010]: }Primary 16D10; Secondary 30G35, 46C05, 46C50.

\newpage
\section{Introduction}
Hilbert spaces over the field of complex numbers are indispensable
for mathematical structure of quantum mechanics \cite{JVN} which in
turn play a great role in molecular, atomic and subatomic
phenomena. The work towards the generalization of quantum
mechanics to bicomplex number system have been recently a topic in
different quantum mechanical models \cite{CV, CVM, BK, Rochon2, Rochon3}.
More specifically, in \cite{GMR2,GMR} the authors made an in depth study of bicomplex
Hilbert spaces and operators acting on them. After obtaining reasonable
results responsible for investigations on finite and infinite
dimensional bicomplex Hilbert spaces and applications to quantum
mechanics \cite{GMR3, GMR4, GeRo, MMR}, they in \cite{GMR} asked for extension of Riesz-Fischer
Theorem and Spectral Theorem on infinite dimensional Hilbert spaces. Recently,
the authors \cite{RKC} have obtained the bicomplex analogue of the Riesz-Fischer
Theorem on infinite dimensional bicomplex Hilbert spaces. They proved that every separable
bicomplex Hilbert space is isometrically isomorphic to the bicomplex analogue of $l^2$.
In this article we obtain a bicomplex version of the Infinite Dimensional Spectral
Decomposition Theorem using bicomplex eigenvalues.

\section{Preliminaries}

This section first summarizes a number of known results
on the algebra of bicomplex numbers, which will be needed
in this paper.  Much more details as well as proofs can
be found in~\cite{Price, Rochon1, Rochon2, Rochon3}.
Basic definitions related to bicomplex modules and scalar
products are also formulated as in~\cite{GMR2, Rochon3}, but
here we make no restrictions to finite dimensions following definitions of \cite{GMR}.


\subsection{Bicomplex Numbers}\label{Bicomplex Numbers}
The set of bicomplex numbers
\begin{align}
\M(2):=\{ w=z_1+z_2\mathbf{i_2}~|~z_1,z_2\in\mathbb{C}(\mathbf{i_1}) \},
\label{2.1}
\end{align}
is a four dimensional algebra extending the reals by two
independent and commuting square roots of $-1$ denoted by $\ii$ and
$\iii$. The product of $\ii$ and $\iii$ defines a hyperbolic unit
$\jj$ such that $\mathbf{j}^2=1$.
They are a particular case of the so-called \textit{Multicomplex Numbers} (denoted $\M(n)$) \cite{Price, GR} and \cite{Vaijac}. In fact, bicomplex numbers $$\M(2)\cong {\rm Cl}_{\Bbb{C}}(1,0) \cong {\rm Cl}_{\Bbb{C}}(0,1)$$
are unique among the complex Clifford algebras (see \cite{BDS,DSS} and \cite{Ryan}) in the sense that this set form a commutative, but not division algebra. An important subset of $\M(2)$ is $\mathbb{D}:= \{ x+y\jj~|~x,y\in\mathbb{R}\}$.
There are three natural involutions of bicomplex numbers, given by conjugations on $\mC (\bo)$ and $\mC (\bt)$ together with their composition. For the sake of notations, we put them as $w^{\dag_1}:=\bar{z}_1+\bar{z}_2\mathbf{i_2};$ $w^{\dag_2}:=z_1-z_2\mathbf{i_2}$
and $w^{\dag_3}:=\bar{z}_1-\bar{z}_2\mathbf{i_2}$. For each conjugation there is a corresponding modulus: $\sqrt{w\cdot w^{\dag_k}}$ for $k=1,2,3$ \cite{Rochon1}. Each of the sets $\mathbb{C}(\ik{k})$ is isomorphic to the field of complex numbers, while $\mathbb{D}\cong {\rm Cl}_{\Bbb{R}}(0,1)$ is the set of so-called \emph{hyperbolic numbers}, also called duplex numbers (see, e.g. \cite {Sob}, \cite {Rochon1}).

The operations of the bicomplex algebra are considerably simplified by
the introduction of two bicomplex numbers $\ee$
and $\eee$ defined as
\begin{equation}
\ee:=\frac{1+\jj}{2},\qquad\eee:=\frac{1-\jj}{2}.\label{2.10}
\end{equation}
In fact $\ee$ and $\eee$ are hyperbolic numbers.
They make up the so-called \emph{idempotent basis}
of the bicomplex numbers. One easily checks that ($k=1,2$)
\begin{equation}
\mathbf{e}_{\mathbf{1}}^2=\ee,
\quad \mathbf{e}_{\mathbf{2}}^2=\eee,
\quad \ee+\eee=1,
\quad \mathbf{e}_{\mathbf{k}}^{\dag_3}=\e{k} ,
\quad \ee\eee=0. \label{2.11}
\end{equation}
Any bicomplex number $w$ can be written uniquely as
\begin{equation}
w = z_1+z_2\iii = z_\hh \ee + z_\hhh \eee , \label{2.12}
\end{equation}
where
\begin{equation}
z_\hh= z_1-z_2\ii \quad \mbox{and}
\quad z_\hhh= z_1+z_2\ii \label{2.12a}
\end{equation}
both belong to $\mathbb{C}(\ii)$.  Note that
\begin{equation}
|w| = \frac{1}{\sqrt{2}}
\sqrt{|z_\hh |^2 + |z_\hhh |^2} \, . \label{norm7}
\end{equation}
The caret notation ($\hh$ and $\hhh$) will be used systematically in
connection with idempotent decompositions, with the
purpose of easily distinguishing different types
of indices.  As a consequence of~\eqref{2.11}
and~\eqref{2.12}, one can check that if
$\sqrt[n]{z_\hh}$ is an $n$th root of $z_\hh$
and $\sqrt[n]{z_\hhh}$ is an $n$th root of $z_\hhh$,
then $\sqrt[n]{z_\hh} \, \ee + \sqrt[n]{z_\hhh} \, \eee$
is an $n$th root of $w$.

The uniqueness of the idempotent decomposition allows the introduction of two projection operators as
\begin{align}
P_1: w \in\M(2)&\mapsto z_\hh \in\C(\ii),\label{2.14}\\
P_2: w \in\M(2)&\mapsto z_\hhh \in\C(\ii).\label{2.15}
\end{align}
The $P_k$ ($k = 1, 2$) satisfies
\begin{equation}
[P_k]^2=P_k, \qquad P_1\ee+P_2\eee=\mathbf{Id}, \label{2.16}
\end{equation}
and, for $s,t\in\M(2)$,
\begin{equation}
P_k(s+t)=P_k(s)+P_k(t),
\qquad P_k(s\cdot t)=P_k(s)\cdot P_k(t) .\label{2.17}
\end{equation}

The product of two bicomplex numbers $w$ and $w'$
can be written in the idempotent basis as
\begin{align}
w \cdot w' = (z_\hh \ee + z_\hhh \eee)
\cdot (z'_\hh \ee + z'_\hhh \eee)
= z_\hh z'_\hh \ee + z_\hhh z'_\hhh \eee .\label{2.20}
\end{align}
Since 1 is uniquely decomposed as $\ee + \eee$,
we can see that $w \cdot w' = 1$ if and only if
$z_\hh z'_\hh = 1 = z_\hhh z'_\hhh$.  Thus $w$ has an inverse
if and only if $z_\hh \neq 0 \neq z_\hhh$, and the
inverse $w^{-1}$ is then equal to
$(z_\hh)^{-1} \ee + (z_\hhh)^{-1} \eee$.  A nonzero $w$ that
does not have an inverse has the property that
either $z_\hh = 0$ or $z_\hhh = 0$, and such a $w$ is
a zero divisor.  Zero divisors make up the
so-called \emph{null cone} $\nc$.  That terminology comes
from the fact that when $w$ is written as in~\eqref{2.1},
zero divisors are such that $z_1^2 + z_2^2 = 0$.
Any hyperbolic number can be written in the
idempotent basis as $x_\hh \ee + x_\hhh \eee$, with
$x_\hh$ and $x_\hhh$ in~$\R$.  We define the set~$\D_+$
of positive hyperbolic numbers as
\begin{equation}
\D_+:= \{ x_\hh \ee + x_\hhh \eee ~|~ x_\hh, x_\hhh \geq 0 \}.
\label{2.21}
\end{equation}
Since $w^{\dag_3} = \bar{z}_\hh \ee + \bar{z}_\hhh \eee$,
it is clear that $w \cdot w^{\dag_3} \in \D_+$ for any
$w$ in $\M(2)$.

\subsection{$\M(2)$-Module and Scalar Product}\label{Module}
The set of bicomplex numbers is a commutative ring.
A module~$M$ defined over the ring of bicomplex numbers is called an
$\M(2)$-\emph{module}~\cite{Rochon3, GMR2, GMR}.

Let $M$ be an $\M(2)$-module. For $k=1, 2$, we define $V_k$
as the set of all elements of the form $\e{k} \ket{\psi}$,
with $\ket{\psi} \in M$.  Succinctly, $V_1:=\eo M$
and $V_2:=\et M$. In fact, $V_k$ is a vector space over $\C(\ii)$ and any element $\ket{v_k} \in V_k$ satisfies
$\ket{v_k} = \e{k} \ket{v_k}$ for $k=1,2$.
It was shown in~\cite{GMR2}
that if~$M$ is a finite-dimensional free $\M(2)$-module, then
$V_1$ and~$V_2$ have the same dimension.

For any $\ket{\psi}\in M$, there exist a unique decomposition
\begin{align}
\ket{\psi} = \ket{v_1}
+ \ket{v_2}, \label{2.31}
\end{align}
where $v_k\in V_k$, $k=1,2$.

It will be useful to rewrite \eqref{2.31} as
\begin{align}
\ket{\psi} = \ket{\psi_\mo}
+ \ket{\psi_\mt} , \label{2.33}
\end{align}
where
\begin{align}
\ket{\psi_\mo} := \eo\ket{\psi}  && \text{and} && \ket{\psi_\mt} := \et\ket{\psi} .\label{2.34}
\end{align}

In fact, the $\M(2)$-module $M$ can be viewed as a vector space $M'$
over $\mC(\bo)$, and $M'=V_1\oplus V_2.$ From a set-theoretical point of view, $M$ and $M'$ are
identical.  In this sense we can say, perhaps improperly,
that the \textbf{module} $M$ can be decomposed into the
direct sum of two vector spaces over $\mC(\bo)$, i.e.\
$M=V_1\oplus V_2.$

\subsubsection{Bicomplex Scalar Product}\label{bicomplex sc}

A \emph{bicomplex scalar product} maps two arbitrary kets
$\ket{\psi}$ and $\ket{\phi}$ into a bicomplex number
$(\ket{\psi}, \ket{\phi})$, so that the following
always holds ($s \in \M(2)$):
\begin{enumerate}
\item $(\ket{\psi}, \ket{\phi} + \ket{\chi})
=(\ket{\psi}, \ket{\phi}) + (\ket{\psi}, \ket{\chi})$;
\item $(\ket{\psi}, s \ket{\phi})
= s (\ket{\psi},\ket{\phi})$;
\item $(\ket{\psi}, \ket{\phi})
= (\ket{\phi}, \ket{\psi})^{\dagger_3}$;
\item $(\ket{\psi}, \ket{\psi})
=0~\Leftrightarrow~\ket{\psi}=0$.
\end{enumerate}
The bicomplex scalar product was defined in~\cite{Rochon3}
where, as in this paper, the physicists' convention is used
for the order of elements in the product.

Property $3$ implies that $(\ket{\psi}, \ket{\psi})\in\D$,
while properties 2 and 3 together imply that
$(s \ket{\psi}, \ket{\phi}) = s^{\dagger_3}
(\ket{\psi},\ket{\phi})$. However, in this work we will also require the
bicomplex scalar product $\(\cdot,\cdot\)$ to be \textit{hyperbolic
positive}, i.e.
\begin{align}
(\ket{\psi},\ket{\psi})\in\D_+,\mbox{
}\forall\ket{\psi}\in M. \label{hyperpositive}
\end{align}
This is a necessary condition if we want to recover the standard quantum mechanics from the bicomplex one (see \cite{GMR3, MMR}).
\begin{definition}
Let $M$ be a $\M(2)$-module and let $(\cdot,\cdot)$
be a bicomplex scalar product defined on $M$. The space
$\{M, (\cdot,\cdot)\}$ is called a $\M(2)$-inner product
space, or bicomplex pre-Hilbert space.  When no confusion
arises, $\{M, (\cdot,\cdot)\}$ will simply be denoted by~$M$.
\end{definition}

In this work, we will sometimes use the Dirac notation
\begin{align}
(\ket{\psi},\ket{\phi})=\braket{\psi}{\phi} \label{2.37}
\end{align}
for the scalar product. The one-to-one correspondence between \emph{bra} $\bra{\cdot}$ and
\emph{ket} $\ket{\cdot}$ can be established from the Bicomplex Riesz Representation Theorem \cite[Th. 3.7]{GMR}.
As in \cite{GeRo}, subindices will be used inside the ket notation. In fact, this is simply a convenient way to deal with the Dirac notation in $V_1$ and $V_2$.
Note that the following projection of a bicomplex scalar product:
\begin{equation}
(\cdot,\cdot)_{\widehat{k}}:=P_k((\cdot,\cdot)):M\times M\longrightarrow \mC(\bo)
\end{equation}
is a \textbf{standard scalar product} on $V_k$, for $k=1,2$. One can easily show \cite{GMR} that
\begin{align}
(\ket{\psi}, \ket{\phi})
&= \ee\P{1}{(\ket{\psi_\mo}, \ket{\phi_\mo})}
+ \eee\P{2}{(\ket{\psi_\mt}, \ket{\phi_\mt})}\\
&=\ee\scalarmath{\ket{\psi_\mo}}{\ket{\phi_\mo}}_\hh+\eee\scalarmath{\ket{\psi_\mt}}{\ket{\phi_\mt}}_\hhh.\\\label{2.36}
&=\ee\braket{\psi_\mo}{\phi_\mo}_\hh+\eee\braket{\psi_\mt}{\phi_\mt}_\hhh.
\end{align}

We point out that a bicomplex scalar product is
\textbf{completely characterized} by the two standard
scalar products $\scalarmath{\cdot}{\cdot}_{\widehat{k}}$ on $V_k$.
In fact, if $\scalarmath{\cdot}{\cdot}_{\widehat{k}}$
is an arbitrary scalar product on $V_k$, for $k=1,2$,
then $\scalarmath{\cdot}{\cdot}$ defined as in \eqref{2.36}
is a bicomplex scalar product on $M$.

From this scalar product, we can define a \textbf{norm}
on the vector space $M'$:
\begin{align}
\big{|}\big{|}\ket{\phi}\big{|}\big{|}
&:= \frac{1}{\sqrt{2}}
\sqrt{\scalarmath{\ket{\phi_\mo}} {\ket{\phi_\mo}}_{\widehat{1}}
+ \scalarmath{\ket{\phi_\mt}}{\ket{\phi_\mt}}_{\widehat{2}}} \notag\\
&=\frac{1}{\sqrt{2}} \sqrt{ \big{|}\ket{\phi_\mo}\big{|}^{2}_{1}
+ \big{|}\ket{\phi_\mt}\big{|}^{2}_{2}} \, .
\label{T-norm}
\end{align}
Here we wrote
\begin{equation}
\big{|}\ket{\phi_\mk }\big{|}_{k}
= \sqrt{\scalarmath{\ket{\phi_\mk}}
{\ket{\phi_\mk}}_{\widehat{k}}} \, ,
\label{normk1}
\end{equation}
where $|\cdot|_k$ is the natural induced norm on~$V_k$.
Moreover,
\begin{equation}
\big{|}\big{|}\ket{\phi}\big{|}\big{|}
= \frac{1}{\sqrt{2}}
\sqrt{\scalarmath{\ket{\phi_\mo}} {\ket{\phi_\mo}}_{\widehat{1}}
+ \scalarmath{\ket{\phi_\mt}}{\ket{\phi_\mt}}_{\widehat{2}}}
= \big{|} \sqrt{\scalarmath{\ket{\phi}}{\ket{\phi}}} \big{|}.
\label{T-norma}
\end{equation}

\begin{definition}
Let $M$ be an $\M(2)$-module and let $M'$ be the associated
vector space. We say that $\|\cdot\|:M\longrightarrow \mathbb{R}$
is a \textbf{$\M(2)$-norm} on $M$ if the following holds:

\smallskip\noindent
1. $\|\cdot\|:M'\longrightarrow \mathbb{R}$ is a norm;\\
2. $\big{\|}w\cdot \ket{\psi}\big{\|}\leq \sqrt{2}
\big{|}w\big{|}\cdot\big{\|}\ket{\psi}\big{\|}$,
$\forall w\in\M(2)$, $\forall \ket{\psi}\in M$.
\label{norm}
\end{definition}
\noindent A $\M(2)$-module with a \textbf{$\M(2)$-norm} is called a
\textbf{normed $\M(2)$-module}. It is easy to check that $\|\cdot\|$ in \eqref{T-norm} is a \textbf{$\M(2)$-norm} on $M$
and that the $\M(2)$-module $M$ is \textbf{complete} with respect
to the following metric on $M$:
\begin{equation}
d(\ket{\phi},\ket{\psi})=\big{|}\big{|}\ket{\phi}-\ket{\psi}\big{|}\big{|}
\end{equation}
if and only if $V_1$ and $V_2$ are complete (see \cite{GMR}).
\begin{definition}
A bicomplex Hilbert space is a $\M(2)$-inner product space $M$ which is complete with respect to the induced $\M(2)$-norm \eqref{T-norm}.
\label{Hilbert}
\end{definition}

\section{Bounded Linear Operators on Bicomplex Hilbert Spaces}

\begin{theorem}
Let $M$ be a bicomplex Hilbert space and $T:M\longrightarrow M $  be a bicomplex linear operator. Then $T$ is continuous if and only if $T$ is bounded.
\label{bounded}
\end{theorem}

\begin{proof}
Suppose $T$ is bounded. Then
\begin{equation}
\big{\|}T( \ket{\phi}) \big{\|}
\le K  \big{\|} \ket{\phi} \big{\|}, \;
\forall \ket{\phi} \in M.  \notag
\end{equation}
Define the projection $T_\mk:M\longrightarrow V_k$ as
\begin{equation}
T_\mk \ket{\phi}
:=  \e{k}T(\ket{\phi}), \;
\forall \ket{\phi} \in M, \; k=1,2. \notag
\end{equation}
Then
\begin{equation}
\big{\|} T_\mk \ket{\phi_\mk} \big{\|}
=\big{\|} \e{k}T(\ket{\phi_\mk}) \big{\|}
\leq \big{\|} T(\ket{\phi_\mk}) \big{\|}
\leq K \big{\|} \ket{\phi_\mk} \big{\|},
\; \forall \ket{\phi_\mk}\in V_k\subset M.
\end{equation}	
Thus $T_\mk$ is bounded on the normed vector space $(V_k,\|\cdot\|)$ for each $k=1,2$.  Therefore $T_\mk$ is continuous for each $k=1,2$. Hence $T=T_\mo + T_\mt$ is continuous.

Conversely, suppose that $T$ is continuous. Then $T_\mo$ and $T_\mt$ being projections of $T$ are continuous. So $T_\mo$ and $T_\mt$ are bounded and therefore $T=T_\mo + T_\mt$ is bounded.
\end{proof}
\begin{remark}
In fact, since a normed $\M(2)$-module $M$ can be viewed as a normed vector space $M'$ over $\mC(\bo)$, all the results concerning normed vector spaces are automatically true for bicomplex Hilbert spaces. In particular, if we denote by $B(M)$ the space of the bicomplex bounded linear operators $T:M\rightarrow M$ with
\begin{equation}
\big{\|}T\big{\|}:=sup\{\big{\|}T(\ket{\phi})\big{\|}: \ket{\phi}\in M, \big{\|} \ket{\phi}\big{\|}\leq 1\},
\end{equation}
then the normed $\M(2)$-module $\{B(M),\big{\|}\cdot\big{\|}\}$ can also be viewed as a normed vector space  $B'(M)\subset B(M')$ over $\mC(\bo)$. Furthermore, if $S:M\rightarrow M$ and $T:M\rightarrow M$ are bicomplex bounded linear operators, then
\begin{equation}
\big{\|}T\circ S\big{\|}\leq \big{\|}T\big{\|}\big{\|}S\big{\|}.
\end{equation}
\end{remark}

By the Bicomplex Riesz Representation Theorem \cite{GMR}, the concept of adjoint operator is well defined over infinite
dimensional bicomplex Hilbert space. However, since the norm cannot be represented directly with a bicomplex scalar product, the standard proof to show that: ``the adjoint operator of a bounded operator is bounded" fail in that case. To obtain a similar result in the bicomplex space, we need to proceed with $M'$.

\begin{lemma}
Let $M$ be a bicomplex Hilbert space. An operator $T:M'\longrightarrow M'$ is a linear operator over the complex Hilbert space $M'$ with the following property: $$T({\bf e_k}\ket{\phi})={\bf e_k} T(\ket{\phi})$$ $\forall \ket{\phi}\in M$ and $k=1,2$ if and only if
$T:M\longrightarrow M$  is a bicomplex linear operator over $M$.
\label{operator}
\end{lemma}

\begin{proof}
From a set-theoretical point of view, $M$ and $M'$ are identical. So, the operator $T$ is automatically well defined over $M$, but we have to prove that $T$ is $\M(2)$-linear. Hence, we have that
\begin{equation}
T(\ket{\phi}+\ket{\psi})=T(\ket{\phi})+T(\ket{\psi})
\end{equation}
and
\begin{equation}
T(z\ket{\phi})=zT(\ket{\phi})
\end{equation}
$\forall \ket{\phi}, \ket{\psi}\in M$ and $\forall z\in\mathbb{C}(\mathbf{i_1})$.
Now, let $w\in\M(2)$ be an arbitrary bicomplex number. Hence, $w=z_\mo \eo + z_\mt \et$ with $z_\mo, z_\mt\in\mathbb{C}(\mathbf{i_1})$, and
\begin{eqnarray}
T(w\ket{\phi}) &=& T(w\ket{\phi})\\
               &=& T((z_\mo \eo + z_\mt \et)\ket{\phi})\\
               &=& T(z_\mo \eo \ket{\phi}) + T(z_\mt \et \ket{\phi})\\
               &=& z_\mo T(\eo \ket{\phi}) + z_\mt T(\et \ket{\phi})\\
               &=& z_\mo \eo T(\ket{\phi}) + z_\mt \et T( \ket{\phi})\\
               &=& w T(\ket{\phi}).
\end{eqnarray}
Conversely, since $\mathbb{C}(\mathbf{i_1})\subset\M(2)$, it is obvious that a bicomplex linear operator on $M$ is automatically a linear operator on $M'$.
\end{proof}

\begin{lemma}
Let $M$ be a bicomplex Hilbert space. Let $T:M\longrightarrow M$ be a bicomplex linear operator over $M$. Then the
adjoint operator $T^*$ for $M$ is the same than the adjoint operator on $M'$.
\label{adjoint}
\end{lemma}

\begin{proof}
Let us work with the idempotent decomposition. Using the projections $T_\mo$ and $T_\mt$, we have that:
\begin{equation}
T(\ket{\phi})=T_\mo(\ket{\phi_\mo})+T_\mt(\ket{\phi_\mt})\mbox{ }\forall \ket{\phi}\in M
\label{T}
\end{equation}
where $T_\mk$ is a linear operator on the complex Hilbert space $V_k$ for $k=1,2$.
Now, let $T_\mk^*$ be the adjoint operator of $T$ on $V_k$, for $k=1,2$, and let's define:
\begin{equation}
T^*(\ket{\phi})=T_\mo^*(\ket{\phi_\mo})+T_\mt^*(\ket{\phi_\mt})\mbox{ }\forall \ket{\phi}\in M.
\label{T^*}
\end{equation}
Since the range of $T_\mk^*$ is in $V_k$ for $k=1,2$, it is clear that $T^*$ is a \textbf{bicomplex linear operator} on $M$.
In fact, by definition, for any $\ket{\psi_{\mk}}, \ket{\phi_{\mk}}\in V_k$ we have that
\begin{equation}
\scalarmath{\ket{\psi_\mk }} {T_\mk \ket{\phi_\mk}}_{\widehat{k}}
= \scalarmath{T_\mk^* \ket{\psi_\mk}} {\ket{\phi_\mk}}_{\widehat{k}} \mbox{ for } k=1,2.
\end{equation}
Hence,
\begin{equation}
\scalarmath{\ket{\psi}} {T \ket{\phi}}= \scalarmath{T^* \ket{\psi}} {\ket{\phi}}
\end{equation}
where
\begin{eqnarray}
\scalarmath{\ket{\psi}}{\ket{\phi}} &=& \braket{\psi}{\phi}\notag \\
&=& \braket{\psi_\mo} {\phi_\mo}_{\widehat{1}}\eo+\braket{\psi_\mt} {\phi_\mt}_{\widehat{2}}\et
\end{eqnarray}
for all $\ket{\psi}, \ket{\phi} \in M$ \cite{GMR}.
Hence, $T^*$ is the adjoint operator of $T$ on $M$. Moreover, by Lemma \ref{operator}, we know that
$T$ is also a \textbf{linear operator} on $M'$. So, if we define the scalar product on $M'$ in this way:
\begin{eqnarray}
\scalarmath{\ket{\psi}}{\ket{\phi}}' &=& \braket{\psi}{\phi}'\notag \\
                                     &:=& \frac{1}{2}\[\braket{\psi_\mo} {\phi_\mo}_{\widehat{1}}+ \braket{\psi_\mt} {\phi_\mt}_{\widehat{2}}\],
\end{eqnarray}
we have that $T^*$ is also the adjoint operator on $M'$ and the norm on $M$ is the same than the norm on $M'$.
\end{proof}

\noindent Using this new approach, we have the following result.

\begin{theorem}
The abjoint operator $T^{*}$ of a bounded operator on a bicomplex Hilbert space $M$ is bounded. Moreover, we have
$\big{\|}T\big{\|}=\big{\|}T^{*}\big{\|}$ and $\big{\|}T^{*}T\big{\|}=\big{\|}T\big{\|}^2$.
\label{adjoint}
\end{theorem}

\begin{proof}
By the Lemma \ref{adjoint}, $T^*$ is also the adjoint operator on $M'$. Therefore, using the analogue results for the complex Hilbert spaces,
we obtain automatically that $T^*$ is continuous on $M'$. Hence, $T^*$ is also continuous on $M$. Moreover, since the norm on $M$ is the same as the norm on $M'$ although they live on the sets with different structures, we obtain also that $\big{\|}T\big{\|}=\big{\|}T^{*}\big{\|}$ and $\big{\|}T^{*}T\big{\|}=\big{\|}T\big{\|}^2$ on $M$ (see \cite[Th. 4.1.1.]{DM}).
\end{proof}

\section{Orthogonal Complements}

In this section, we explain more precisely the relationship between $M$ and $M'$. First, it is easy to show that $V_1$ is orthogonal to $V_2$ in $(M,\scalarmath{\cdot}{\cdot})$ and $(M',\scalarmath{\cdot}{\cdot}')$.
In fact, $V_1^{\bot}=V_2$. Therefore, the same symbol $\bot$ can used for $M$ and $M'$, and we have
\begin{equation}
M=V_1\oplus V_1^{\bot}=V_1\oplus V_2=M'.
\end{equation}
However, this is not the case for the subspace $V$.
Let $\{\ket{\psi_1} \dots \ket{\psi_n} \dots\}$ be a Schauder
$\M(2)$-basis associated with the bicomplex Hilbert space
$\{M, (\cdot,\cdot)\}$.  That is, any element $\ket{\psi}$
of $M$ can be written as
\begin{equation}
\ket{\psi} = \sum_{n=1}^{\infty} w_n\ket{\psi_n} ,\label{2.22}
\end{equation}
with $w_n \in \M(2)$.  As was shown in~\cite{Rochon3} for the
finite-dimensional case, an important subset $V$ of $M$ is
the set of all kets for which all $w_n$ in \eqref{2.22}
belong to $\C(\ii)$. It is obvious that $V$ is a non-empty
normed vector space over complex numbers with Schauder basis
$\{\ket{\psi_1} \dots \ket{\psi_n} \dots\}$.

From Theorem~\cite[Th. 3.11]{GMR} we see that if
$\{\ket{\psi_1} \dots \ket{\psi_n} \dots\}$ is an \textbf{orthonormal}
Schauder $\M(2)$-basis and
\begin{equation}
\sum_{n=1}^{\infty} (\ee {z_{n\hh}}+\eee {z_{n\hhh}})\ket{\psi_n}
\notag
\end{equation}
converges in $M$, then the series
\begin{equation}
\sum_{n=1}^{\infty} |\ee {z_{n\hh}}+\eee {z_{n\hhh}}|^2
\notag
\end{equation}
converges in~$\R$. In particular,
$\sum_{n=1}^{\infty} |z_{n\h{k}}|^2$ also converges. Hence
$\sum_{n=1}^{\infty} {z_{n\h{k}}}\ket{\psi_n}$ converges and this allows
to define projectors $P_1$ and $P_2$ from $M$ to $V$ as
\begin{equation*}
\P{k}{\ket{\psi}} := \sum_{n=1}^{\infty} {z_{n\h{k}}}\ket{\psi_n},
\qquad k=1, 2.
\end{equation*}
Therefore, any $\ket{\psi} \in M$ can be decomposed
uniquely as
\begin{equation}
\ket{\psi} = \ee \P{1}{\ket{\psi}}
+ \eee \P{2}{\ket{\psi}}.
\label{2.23}
\end{equation}

As in the finite-dimensional case \cite{Rochon3},
one can easily show that ket projectors
and idempotent-basis projectors (denoted with the
same symbol) satisfy the following, for $k=1,2$:
\begin{align}
\P{k}{s \ket{\psi} + t \ket{\phi}}
= \P{k}{s} \P{k}{\ket{\psi}}
+ \P{k}{t} \P{k}{\ket{\phi}} .\label{2.24}
\end{align}

\begin{definition}
Let $\{ \ket{\psi_n} \}$ be an orthonormal Schauder $\M(2)$-basis of $M$ and let $V$ be
the associated vector space.  We say that a scalar product
is $\mC(\bo)$-closed in $V$ if
$\forall \ket{\psi},\ket{\phi}\in V$ implies
$(\ket{\psi},\ket{\phi})\in\mC(\bo)$.
\end{definition}
We recall (see \cite{GMR}) that if the scalar product is $\mC(\bo)$-closed in $V$ then
the inner space $(V,||\cdot||)$ is closed in $M$. Hence, since any closed linear subspace
of a Hilbert space satisfy the Projection Theorem \cite{Hansen}, we have that
\begin{equation}
M'=V\oplus V^{\bot}
\end{equation}
when the scalar product is $\mC(\bo)$-closed under $V$. In this case, it is easy to verify that the orthogonal complement of $V$ for $(M',\scalarmath{\cdot}{\cdot}')$
is
\begin{equation}
V^{\bot}=\oa \ee \ket{\psi}- \eee \ket{\psi} : \ket{\psi}\in V\fa.
\end{equation}

This is not the case for $(M,\scalarmath{\cdot}{\cdot})$ since the orthogonal complement of $V$ is $\{0\}$. In fact, since $M$ is not a Hilbert space, the Projection Theorem cannot be applied.

\noindent Finally, using \eqref{2.23}, if we define
\begin{equation}
V_1^{\dagger_2}:=\oa  \ee \P{2}{\ket{\psi}}+ \eee \P{1}{\ket{\psi}} : \ket{\psi}\in V_1\fa=\oa \eee \P{1}{\ket{\psi}} : \ket{\psi}\in V_1\fa
\end{equation}
where $\dagger_2$ is used as the natural extension of the conjugate $\dagger_2$ in $\M(2)$,
we obtain that $V_1^{\dagger_2}=\eee V=V_2=V_1^{\bot}$ and
\begin{equation}
M=V_1\oplus V_1^{\dagger_2}=V_1\oplus V_2=M'.
\end{equation}

\begin{remark}
This definition of $\dagger_2$ is \textbf{universal} for any element inside a bicomplex Hilbert space with an orthonormal Schauder $\M(2)$-basis, and satisfy the following properties:
\begin{enumerate}
\item $(\ket{\phi}^{\dagger_2})^{\dagger_2}=\ket{\phi}$;
\item $(\ket{\phi}\pm\ket{\psi})^{\dagger_2}=\ket{\phi}^{\dagger_2}\pm\ket{\psi}^{\dagger_2}$;
\item $(w\ket{\phi})^{\dagger_2}=w^{\dagger_2}\ket{\phi}^{\dagger_2}$
\end{enumerate}
$\forall \ket{\phi},\ket{\psi}\in M$ and $\forall w\in\M(2)$.
\end{remark}

\section{Bicomplex Compact Operators}

\begin{definition}
A bicomplex linear operator $T$ on a bicomplex Hilbert space $M$ is called a \textit{bicomplex compact operator} if, for every
bounded sequence $\{\ket{\phi_n}\}$ in $M$, the sequence $\{T(\ket{\phi_n})\}$ contains a convergent subsequence.
\label{compact}
\end{definition}

We can verify easily that the collection of all bicomplex compact operators on a bicomplex Hilbert space $M$ is an $\M(2)$-module.
Moreover, since $\{M,\scalarmath{\cdot}{\cdot}\}$ and $\{M',\scalarmath{\cdot}{\cdot}'\}$ share the same topology, we obtain the following result as a direct consequence of Lemma \ref{operator}.

\begin{theorem}
Let $T$ be a bicomplex linear operator on a bicomplex Hilbert space $M$. Then $T$ is a bicomplex compact operator if and only if $T$ is a compact operator on $M'$.
\label{compact}
\end{theorem}

Therefore, many results are automatically true for bicomplex compact operators. We can cite the following as examples (see \cite{DM}):

\begin{corollary}
Bicomplex compact operators are bounbed.
\end{corollary}

\begin{corollary}
The adjoint of a bicomplex compact operator is compact.
\end{corollary}

\noindent However, using Theorem \ref{compact}, it is not possible to obtain a bicomplex version of the Spectral Theorem for Self-Adjoint Compact Operators.
In this way, we obtain the standard Spectral Theorem for $M'$ where the eigenvalues are complex. To obtain a bicompex version, we need to come back
to the idempotent decomposition with the following lemma.

\begin{lemma}
Let $T$ be a bicomplex linear operator on a bicomplex Hilbert space $M$. Then $T$ is a bicomplex compact operator if and only if
$T_\mk$ is compact on $V_k$ for $k=1,2$.
\label{compact_id}
\end{lemma}

\begin{proof}
Let $\{\ket{\phi_{n\mk}}\}$ be an arbitrary bounded sequence in $\{V_k,\scalarmath{\cdot}{\cdot}_{\widehat{k}}\}$ for $k=1,2$.
From the definition of the $\M(2)$-norm (see \cite{GMR}) we have that
\begin{equation}
\big{|} \ket{\psi_\mk} \big{|}_{k}
= \sqrt{2} \big{\|} \ket{\psi_\mk} \big{\|} , \;\;
\forall \ket{\psi_\mk} \in V_k . \notag
\end{equation}
Hence, $\ket{\phi_{n\mk}}$ is also a bounded sequence in $\{M,\scalarmath{\cdot}{\cdot}\}$ for $k=1,2$.
Now, using the definition of a bicomplex compact operator, we have that the sequence $\{T(\ket{\phi_{n\mk}})\}$ contains
a convergent subsequence for $k=1,2$. Hence,
$T_\mk$ is compact on $V_k$ since $T_\mk(\ket{\phi_n})=T(\ket{\phi_{n\mk}})$ for $k=1,2$.
Conversely, if $\{\ket{\phi_n}\}$ is an arbitrary bounded sequence in $M$, then $\{\ket{\phi_{n\mk}}\}$ is also
a bounded sequence in $\{V_k,\scalarmath{\cdot}{\cdot}_{\widehat{k}}\}$ for $k=1,2$ since
\begin{equation}
\big{|} \ket{\psi_\mk} \big{|}_{k}
\leq \sqrt{2} \big{\|} \ket{\psi} \big{\|} , \;\;
\forall \ket{\psi_\mk} \in V_k. \notag
\end{equation}
Therefore, if we apply successively the definition of compact operators for $k=1,2$,
we find a convergent subsequence $\{T_\mk(\ket{\phi_{n_l\mk}})\}$ for $k=1,2$.
Hence, since $T(\ket{\phi_{n_l}})=T_\mo(\ket{\phi_{n_l\mo}})+T_\mt(\ket{\phi_{n_l\mt}})$,
the sequence $\{T(\ket{\phi_n})\}$ contains a convergent subsequence.
\end{proof}

We are now ready to prove the infinite dimensional bicomplex spectral decomposition theorem.

\begin{theorem}
Let $T$ be a self-adjoint, bicomplex compact operator on a separable infinite dimensional bicomplex Hilbert space $M$.
Then there exist in $M$ an orthonormal (Schauder) $\M(2)$-basis $\{ \ket{\psi_n} \}$ consisting of eigenkets for $T$.
Moreover, for every $\ket{\phi}\in M$,
\begin{equation}
T(\ket{\phi})=\sum_{n=1}^{\infty}\lambda_n \scalarmath{\ket{\psi_n}}{\ket{\phi}}\ket{\psi_n}
\end{equation}
or using the bra-ket notation,
\begin{equation}
T=\sum_{n=1}^{\infty}\lambda_n \ket{\psi_n}\bra{\psi_n}
\end{equation}
where $\lambda_n$ is the eigenvalue corresponding to $\ket{\psi_n}$.
\end{theorem}

\begin{proof}
By Lemma \ref{compact_id}, the projection operator $T_\mk$ is compact for $k=1,2$. Moreover, using \eqref{T} and \eqref{T^*}, we have that
$T_\mk$ is also self-adjoint for $k=1,2$. Hence, the classical Spectral Decomposition Theorem holds for $T_\mk$ on $V_k$
since $V_k$ is a separable infinite dimensional Hilbert space for $k=1,2$ (see \cite{RKC}). From this result, we get the orthonormal sets
$\{ \ket{\psi_{n\mo}} \}$ and $\{ \ket{\psi_{n\mt}} \}$ of eigenvectors of $T_\mo$ and $T_\mt$ respectively.
Now, let
\begin{equation}
\ket{\psi_n}:=\ket{\psi_{n\mo}}+\ket{\psi_{n\mt}}\mbox{ }\forall n\in\mathbb{N}^{*}.
\end{equation}
It is easy to show that $\{ \ket{\psi_n} \}$ an orthonormal $\M(2)$-basis.
Now, let $\lambda_n$ be the eigenvalue of $T$ associated with $\ket{\psi_n}$ so that
\begin{equation}
T(\ket{\psi_n})=\lambda_n\ket{\psi_n}.
\label{EQN}
\end{equation}
Since $\forall p\in\mathbb{N}^{*}$ we have that,
\begin{align}
\[\sum_{n=1}^{\infty}\lambda_n\ket{\psi_n}\bra{\psi_n}\]\ket{\psi_p}&=\sum_{n=1}^{\infty}\lambda_n\ket{\psi_n}\braket{\psi_n}{\psi_p}\\
&=\sum_{n=1}^{\infty}\lambda_n\ket{\psi_n}\delta_{np}\\
&=\lambda_p\ket{\psi_p}\\
&=T(\ket{\psi_p})\ \mbox{by }\eqref{EQN}
\end{align}
then, by linearity and continuity of the operator $T$, we obtain the following conclusion:
$$T=\sum_{n=1}^{\infty}\lambda_n \ket{\psi_n}\bra{\psi_n}.$$
\end{proof}

\section{Acknowledgments}

DR is grateful to the Natural Sciences and Engineering Research Council of Canada for financial support. The authors are grateful to Dr Romesh Kumar for the discussions with DR at the very initial stages of this work.

\bibliographystyle{amsplain}

\begin{thebibliography}{99}

\bibitem{BDS} F. Brackx, R. Delanghe, F. Sommen, Clifford analysis, Pitman, London, 1982.

\bibitem{CV} H. M. Campos et V. V. Kravchenko, \textit{Fundamentals of Bicomplex Pseudoanalytic Function Theory: Cauchy Integral Formulas,
 Negative Formal Powers and Schrödinger Equations with Complex Coefficients}, Complex Anal. Oper. Theory, (to appear).

\bibitem{CVM} H. M. Campos, V. V. Kravchenko and L. M. M\'endez, \textit{Complete Families of Solutions for the Dirac Equation Using
Bicomplex Function Theory and Transmutations}, Adv. Appl. Clifford Algebras $\textbf{22}$, No. 3 (2012), 577--594.

\bibitem{BK} W. E. Baylis and J. D. Keselica, \textit{The Complex Algebra of Physical Space: A Framework for Relativity}, Adv. Appl.
Clifford Algebras $\textbf{22}$, No. 3 (2012), 537--561.

\bibitem{DM} L. Debnath and P. Mikusi\'nski, Introduction to Hilbert Spaces with Applications 2nd ed., Academic Press, London, 1999.

\bibitem{DSS} R. Delanghe, F. Sommen, V. Soucek, Clifford analysis and spinor valued functions,
Kluwer Acad. Publ., Dordrecht, 1992.

\bibitem{GMR2} R. Gervais Lavoie, L. Marchildon and D. Rochon,
\textit{Finite-dimensional bicomplex Hilbert spaces},
Adv. Appl. Clifford Algebr. \textbf{21}, No. 3 (2011), 561--581.

\bibitem{GMR} R. Gervais Lavoie, L. Marchildon and D. Rochon, \textit{Infinite dimensional Hilbert spaces},
Ann. Funct. Anal. \textbf{1}, No. 2 (2010), 75--91.

\bibitem{GMR3} R. Gervais Lavoie, L. Marchildon and D. Rochon, \textit{The Bicomplex Quantum Harmonic Oscillator},
Nuovo Cimento B. \textbf{125}, No. 10 (2010), 1173--1192.

\bibitem{GMR4} R. Gervais Lavoie, L. Marchildon and D. Rochon, \textit{Hilbert Space of the Bicomplex Quantum Harmonic Oscillator},
AIP Conference Proceedings \textbf{1327}, (2011), 148--157.

\bibitem{GeRo} R. Gervais Lavoie and D. Rochon, \textit{The Bicomplex Heisenberg Uncertainty Principle},
Theoretical Concepts of Quantum Mechanics, ISBN 978-953-51-0088-1, InTech Book, 2012, 39--64.

\bibitem{Hansen} V. L. Hansen, Functional Analysis: Entering Hilbert Space,
World Scientific, Singapore, 2006.

\bibitem{MMR} J. Mathieu, L. Marchildon and D. Rochon, \textit{The bicomplex quantum Coulomb potential problem},
(2012) \texttt{arXiv:1207.0766}.

\bibitem{JVN} J. von Neumann, Mathematical Foundations of Quantum
Mechanics, Princeton University Press, Princeton, 1955.

\bibitem{Price} G. B. Price, An Introduction to Multicomplex Spaces and Functions,
Marcel Dekker, 1991.

\bibitem{GR} V. Garant-Pelletier and D. Rochon, \textit{On a generalized Fatou-Julia theorem in multicomplex spaces}, Fractals \textbf{17},
No. 3 (2009), 241--255.

\bibitem{RKC} D. Rochon, R. Kumar and K. S. Charak, \textit{Bicomplex Riesz-Fischer Theorem}, \texttt{arXiv:1109.3429}.

\bibitem{Rochon1} D. Rochon and M. Shapiro, \textit{On algebraic properties of bicomplex and hyperbolic numbers},
Ann. Univ. Oradea, Fasc.\ Matematica \textbf{11}, (2004), 71--110.

\bibitem{Rochon2} D. Rochon and S. Tremblay,
\textit{Bicomplex quantum mechanics: I. The generalized
Schr\"{o}dinger equation}, Adv. Appl. Clifford Algebr. \textbf{14}, No. 2 (2004), 231--248.

\bibitem{Rochon3} D. Rochon and S. Tremblay,
\textit{Bicomplex quantum mechanics: II. The Hilbert space},
Adv. Appl. Clifford Algebr. \textbf{16}, No. 2 (2006), 135--157.

\bibitem{Ryan} J. Ryan, \textit{Complexified Clifford Analysis}, Complex Variables {\bf 1}, No. 1 (1982), 119--149.

\bibitem{Sob} G. Sobczyk, \textit{The hyperbolic number plane}, Coll. Maths. Jour. \textbf{26}, No. 4 (1995), 268--280.

\bibitem{Vaijac} A. Vaijac and M. B. Vaijac, \textit{Multicomplex hyperfunctions},
Complex Var. Elliptic Equ. $\textbf{57}$, Nos. 7-8 (2012), 751--762.

\end{thebibliography}

\end{document}